\documentclass{article}
\usepackage{graphicx} 
\usepackage[utf8]{inputenc}
\usepackage{amsmath}  
\usepackage{amsfonts}
\usepackage{amssymb}

\usepackage{bbm}
\usepackage{authblk}

\usepackage{geometry}
\newcommand{\comment}[1]{}

\usepackage{chngcntr}

\counterwithin*{equation}{section}
\counterwithin*{equation}{subsection}

\usepackage{amsmath}
\usepackage{amsthm}

\theoremstyle{plain}

\theoremstyle{plain}
\newtheorem{theorem}{Theorem}[]
\newtheorem{lemma}[theorem]{Lemma}
\newtheorem{corollary}[theorem]{Corollary}

\newtheorem{proposition}[theorem]{Proposition}
\newtheorem*{definition}{Definition}

\newtheorem*{remarks}{Remarks}

\usepackage{graphicx} 

\newcommand\Nat{\mathbb{N}}
\newcommand\RCA{\mathsf{RCA}}

\newcommand\RT{\mathsf{RT}}

\newcommand\W{\mathrm{W}}
\newcommand\sW{\mathrm{sW}}

\newcommand\WO{\mathrm{WO}}

\newcommand\WOP{\mathsf{WOP}}
\newcommand\ART{\mathsf{ART}}
\newcommand\ORT{\mathsf{ORT}}

\newcommand\Q{\mathsf{Q}}
\newcommand\PP{\mathsf{P}}
\newcommand\TC{\mathsf{TC}}
\newcommand\LPO{\mathsf{LPO}}
\newcommand\ECT{\mathsf{ECT}}

\newcommand\IsFinite{\mathsf{IsFinite}}

\title{Weihrauch reducibility between Ramsey-type theorems and well-ordering principles at the level of $\Sigma^0_2$-induction: A pilot study}
\author[1]{Lorenzo Carlucci}
\author[2]{Giordano Celli}
\affil[1,2]{Sapienza University of Rome, Department of Mathematics}
\affil[1]{lorenzo.carlucci@uniroma1.it}
\affil[2]{celli.1873376@studenti.uniroma1.it}

\begin{document}

\maketitle

\begin{abstract}
We study the relations under Weihrauch reducibility of the well-ordering preservation principle for the 
operator $X \mapsto X^\omega$ and the Ordered Ramsey Theorem. Both principles are known to be 
equivalent to $\Sigma^0_2$-induction in Reverse Mathematics. 

We show that the Ordered Ramsey Theorem is Weihrauch-equivalent to the parallel product of 
the well-ordering preservation principle for the 
operator $X \mapsto X^\omega$ and the Eventually Constant Tail principle. 

By previous work from Pauly, Pradic and Soldà, the Ordered Ramsey Theorem is known to be Weihrauch-equivalent to the parallel product of 
the Eventually Constant Tail principle and the parallelization of the jump of the Limited Principle of 
Omniscience. We show that the latter pinciple and the well-ordering preservation principle for $X \mapsto X^\omega$ are Weihrauch-incomparable. 
\end{abstract}

\section{Introduction and Motivation}

A {\em well-ordering preservation principle} (or {\em well-ordering principle} -- WOP for short) states that some operator on linear orderings preserves well-ordering. A {\em Ramsey-type theorem} states that some class of colorings is partition-regular. Both families of principles are well investigated in Reverse Mathematics \cite{Dza-Mum:22}. Many prominent systems can be characterized in terms of WOPs (see \cite{Rat-Wei:11, Mar-Mon:11}) and of Ramsey-type theorems (see \cite{Dza-Mum:22}). Usually the equivalence of the latter two is obtained as an indirect by-product. We refer to \cite{Car-Mai-Zda:25} for examples and discussion. 

In the present paper we study from the viewpoint of Weihrauch reducibility the well-ordering principle for the operator $X \mapsto X^\omega$ and the Ordered Ramsey Theorem. This is one way of making precise the question from \cite{Car-Zda:12} whether {\em direct proofs} exist, in one direction or the other, between Ramsey-type principles and WOPs.

The motivation for the present work is twofold. First, we continue the line of research initiated in \cite{Car-Mai-Zda:25} concerning the possibility of establishing Weihrauch reductions of well-ordering preservation principles to combinatorial theorems from Ramsey Theory. For example it is shown in \cite{Car-Mai-Zda:25} that the WOP for the operator $X \mapsto \omega^X$ is Weihrauch reducible to Ramsey's Theorem for $2$-colorings of $3$-subsets of $\Nat$ as well as to Hindman's Theorem restricted to finite sums of $3$ terms. The initial motivation for the present work comes from the observation in \cite{Car-Mai-Zda:25} that a similar proof shows that the WOP for the operator $X \mapsto X^\omega$ reduces to Ramsey's Theorem for finite colorings of $2$-subsets. The latter result is not optimal from the point of view of Reverse Mathematics, since the WOP in question is known to be equivalent to $\Sigma^0_2$-induction \cite{Uft:23}, while Ramsey's Theorem for pairs and arbitrarily many colors implies the stronger principle of $\Sigma^0_3$-bounding \cite{Hir:87}.  It was therefore natural to ask whether the same proof-technique from \cite{Car-Mai-Zda:25} could be applied to reduce the WOP for $X \mapsto X^\omega$ to a Ramsey-type statement known to be equivalent to $\Sigma^0_2$. The Ordered Ramsey Theorem introduced in \cite{KMPS:19} is one such candidate. The fact that the latter principle was studied from the point of view of Weihrauch complexity in \cite{Pau-Pra-Sol:24} has been our main motivation for focusing, among the known combinatorial principles equivalent to $\Sigma^0_2$-induction, on the Ordered Ramsey Theorem. The analysis in~\cite{Pau-Pra-Sol:24} gave some indication on how to obtain reductions in the  direction, from the WOP for $X\mapsto X^\omega$ to the Ordered Ramsey Theorem. Note that reductions {\em from} WOPs {\em to} Ramsey-type statements are not to be found in the previous literature.

\section{Background}

We introduce the needed technical notions, the main principles of interest and the relevant results from previous literature. 

The principles considered in this paper can be expressed as {\em instance-solution problems}, i.e.~as partial functions $\PP: A \to \mathcal{P}(B)$ for some sets $A$ and $B$. The elements of the domain of $\PP$ are called {\em instances} of $\PP$. If $a\in A$ is an instance of $\PP$ then each $b \in \PP(a)\subseteq B$ is called a {\em solution of} $\PP$ {\em to} $a$, or a $\PP$-solution to $a$. It is well-kown that $\forall\exists$ theorems of the form $\forall x (F(x) \to \exists y G(x,y))$ of interest in Reverse Mathematics can be naturally expressed as problems (see \cite{Dza-Mum:22}. 

\subsection{Weihrauch reducibility}

Weihrauch reducibility has become in later years of major interest in Computable Mathematics and in Reverse Mathematics. We give the basic definitions.

\begin{definition}[(strong) Weihrauch reducibility]
Let $\PP$ and $\Q$ be problems. 
\begin{enumerate}
\item $\PP$ is Weihrauch-reducible to $\Q$ (denoted $\PP \leq_\W \Q$), if there exist Turing functionals $H$ and $K$ such that: 
if $X$ is any $\PP$-instance then $H(X)$ is a $\Q$-instance, and if $\hat{Y}$ is any $\Q$-solution to $H(X)$ then $K(X \oplus \hat{Y})$
is a $\PP$-solution to $X$.
\item $\PP$ is strongly Weihrauch-reducible to $\Q$ (denoted $\PP \leq_\sW \Q$), if there exist Turing functionals $H$ and $K$ such that: 
if $X$ is any $\PP$-instance then $H(X)$ is a $\Q$-instance, and if $\hat{Y}$ is any $\Q$-solution to $H(X)$ then $K(\hat{Y})$
s a $\PP$-solution to $X$.
\end{enumerate}
\end{definition}

We abbreviate by $\PP \equiv_\W \Q$ the fact that $\PP \leq_\W \Q$ and $\Q \leq_\W \PP$. Similarly for $\PP \equiv_\sW \Q$. 
We abbreviate by $\PP <_W \Q$ the fact that $\PP \leq_\W \Q$ and $\Q \nleq_\W \PP$. Similarly for $<_\sW$. 

Some of our results are expressed in terms of a number of structural operations on problems, which respect the quotient to Weihrauch degrees. 

\begin{definition}[Parallel product and finite parallelization]
Let $\PP$ and $\Q$ be problems. 
\begin{enumerate}
\item The parallel product of $\PP$ and $\Q$ (denoted $\PP\times \Q$) is the problem whose instances are all pairs $(X_0, X_1)$ 
where $X_0$ is a $\PP$-instance and $X_1$ is a $\Q$-instance, with the solutions being all pairs $(Y_0, Y_1)$ such that $Y_0$ is a $\PP$-solution to $X_0$ and $Y_1$ is a $\Q$-solution to $X_1$.
\item The finite parallelization of a problem $\PP$ (denoted $\PP^*$), has the power to solve an arbitrary finite number of instances of $\PP$, provided that number is given as part of the input.
\end{enumerate}
\end{definition}

For every problem $\PP$ and $n\geq 1$ we write $\PP^n$ to mean $\underbrace{\PP \times \dots \times \PP}_{n \small{\mbox{ times}}}$.

The following problems have emerged as useful benchmarks in Weihrauch Complexity. 

\begin{enumerate}
\item $\TC_\Nat: \subseteq \Nat^\Nat \Rightarrow \Nat$ (totalization of closed choice on N) is the problem that takes as input an enumeration e of any subset of $\Nat$ (hence now we allow the possibility that $ran(e) = \Nat$) and such that, for every $n \in \Nat$, $n \in TCN(e)$ if and only if $n  \notin ran(e)$ or $ran(e) = \Nat$.

\item $\LPO: 2^\Nat \to \{0,1\}$ (limited principle of omniscience) takes as input any infinite binary string $p$ and outputs $0$ if and only if $p = 0^\Nat$. 

\item $\LPO': \subseteq 2^\Nat \to \{0,1\}$: takes as input (a code for) an infinite sequence $\langle p_0,p_1,\dots\rangle$ of binary strings such that the function $p(i) = \lim_{s\to\infty} p_i(s)$ is defined for every $i \in \Nat$, and outputs $\LPO(p)$.

\item Define the problem $\IsFinite_\mathbb S$ as follows: an instance of $\IsFinite_\mathbb S$ is a binary sequence $p\in 2^\mathbb N$, and the solution $\IsFinite_\mathbb S(p)\in \mathbb S$ for $p$ is $\top$ if $p$ contains finitely many occurrences of $1$, and $\bot$ otherwise. Here, $\mathbb S$ denotes the Sierpinski space $\{\top, \bot\}$, where $\top$ is coded by binary sequences containing at least one occurrence of $1$, and $\bot $ is coded by $0^\mathbb N$.

\end{enumerate}

The notation $\LPO'$ is justified by the fact that the definition of $\LPO'$ is a special case of the definition of jump as given in \cite{Bra-Ghe-Pau:21} applied to $\LPO$. As noted in \cite{Pau-Pra-Sol:24}, $\LPO'$ corresponds to the power of answering a single binary $\Sigma^0_2$
-question.

\subsection{Well-ordering principles}

A well-ordering preservation principle states that some operator on linear orders preserves well-ordering. In the present paper we are interested in a particular WOP based on the operator $X \mapsto X^\omega$ corresponding to ordinal exponentiation to the power $\omega$. When referring to a linear ordering $\mathcal{X} = (X, \leq_X)$ we systematically use $X$ also for $(X, \leq_X)$. 

The following definition is from \cite{Hir:94}. 

\begin{definition}[The operator $X \mapsto X^\omega$]
If $X$ is a linear ordering we define $X^\omega$ as the set of all finite sequences of the form $\langle (b_0, a_0), \dots, (b_n,a_n)\rangle$
such that  for all $i \leq n$, $b_i \in \Nat$, $b_0 > b_1 > \dots > b_n$ and $0 \neq a_i \in X$. We let $<_{X^\omega}$ be the lexicographic ordering on $X^\omega$: If $\sigma$ extends $\tau$ then $\sigma >_{X^\omega} \tau$. If $j$ is least integer such that $(b_j, a_j) = \sigma_j \neq \tau_j = (b'_j, a_j)$ and either $b_j > b'_j$ or both $b_j = b'_j$ and $a_j >_X a'_j$ then $\sigma> \tau$. Otherwise $\tau >_{X^\omega} \sigma$.
\end{definition}

If $X$ is an ordinal $\alpha$, the element $\langle (b_0, a_0), \dots, (b_n,a_n)\rangle$ of $\alpha^\omega$ corresponds to the ordinal 
$\alpha^{b_0}\cdot a_0 + \dots + \alpha^{b_n}\cdot a_n$. 

In the Reverse Mathematics literature $\WOP(X\mapsto X^\omega)$ denotes the principle $\forall X (\WO(X) \to \WO(X^\omega))$ where $\WO(X)$ is any canonical $\Pi^1_1$-formula expressing that $X$ is well-ordered. Uftring \cite{Uft:23} proved that this principle is 
equivalent to $\Sigma^0_2$-induction. 

As patent from the definition of $X \mapsto X^\omega$, we can view the elements of $X^\omega$ as {\em terms} involving elements from $X$ as {\em sub-terms}. If $\sigma$ is a sequence in $X^\omega$ and $\sigma'$ is a sequence in $X$, we say that $\sigma'$ {\em is contained in} $\sigma$ if $\sigma'$ consists of sub-terms of elements of $\sigma$ in the same relative order, i.e.: $\sigma'_i$ is a sub-term of $\sigma_{t_i}$ for some $t_i$ and if $i < j$ then $t_i \leq t_j$. The notion makes sense for all other WOPs considered in this paper, which are based on operators weaker than $X \mapsto X^\omega$, as well as for most WOPs considered in the Reverse Mathematics literature. 

In order to study the Weihrauch complexity of WOPs we give the following definition.

\begin{definition}[WOP Problem]
Let $F$ be an operator on linear orders. The problem $\WOP(X\mapsto F(X))$ is the problem that takes as an infinite decreasing sequence $\sigma$ in $F(X)$ and outputs an infinite decreasing sequence $\sigma'$ in $X$ contained in $\sigma$. 
\end{definition}

\subsection{Ramsey-type theorems}

We here introduce the Ramsey-type principles of interest to the present paper, i.e., the Ordered and the Additive Ramsey Theorems. 

Fix a linear order $(X,<_X)$. For every poset $(P,<_P)$, we call a colouring $c : [X]^2 \to P$ 
{\em right-ordered} if we have $c(x,y) \leq_P c(x,y')$ when $x <_X y \leq_X y'$. 

\begin{definition}[Ordered Ramsey Theorem]
The problem $\ORT_\Nat$ takes as input a finite poset $(P,\leq_P)$ and a right-ordered colouring $c : [\Nat]^2 \to P$,
and outputs an infinite $c$-homogeneous set $\subseteq\Nat$.
\end{definition}

\begin{definition}[Additive Ramsey Theorem]
The problem $\ART_\Nat$ takes as input a finite semigroup $S$ and an additive colouring $c : [\Nat]^2 \to S$, and
outputs an infinite $c$-homogeneous set $\subseteq\Nat$.
\end{definition}

In Reverse Mathematics the $\forall\exists$ theorems naturally corresponding to $\ORT_\Nat$ and $\ART_\Nat$ in Second-Order Arithmetic  were introduced in \cite{KMPS:19} and proved equivalent to $\Sigma^0_2$-induction over $\RCA_0$. The Weihrauch complexity of the problems $\ORT_\Nat$ and $\ART_\Nat$ is studied in \cite{Pau-Pra-Sol:24}.

Finally we will make use of the principle called Eventually Constant Tail, introduced in \cite{DHHPPK:20}. 

\begin{definition}[Eventually Constant Tail]
$\ECT$ is the problem whose instances are pairs $(n,f) \in \Nat \times \Nat^\Nat$ such that $f :\Nat \to n$ is a colouring of the natural numbers with $n$ colours, and such that, for every instance $(n,f)$ and $b \in \Nat$, $b \in \ECT(n,f)$ if and only if $\forall x >b\exists y >x(f(x) =f(y))$. 
\end{definition}

In other words, $\ECT$ is the problem that, upon being given a function $f$ of the integers with finite range, outputs a $b$ such that, after that $b$, the palette of colours used is constant (hence its name, which stands for eventually constant palette tail). In Reverse Mathematics the principle corresponding to $\ECT$ is equivalent to $\Sigma^0_2$-induction \cite{DHHPPK:20}.

\subsection{Known relations and overview of the paper}

The following was recently proved by Pauly, Pradic and Soldà \cite{Pau-Pra-Sol:24}.

\begin{theorem}[Pauly, Pradic, Soldà]\label{thm:pauly}
$\ORT_\Nat \equiv_\W \ART_\Nat \equiv_\W \ECT \times (\LPO')^*.$
\end{theorem}

As mentioned above, in Reverse Mathematics the principles corresponding to $\ORT_\Nat$, $\ART_\Nat$ and $\ECT$ are 
all equivalent to $\Sigma^0_2$ induction \cite{DHHPPK:20}. 

Other known results concerning the Weihrauch complecity of principles of interest in the present paper are the following. 
The first concerns the relation between $\ECT$ and $\TC_\Nat$ and is \cite[Theorem 9]{DHHPPK:20}.

\begin{lemma}[Davis et alii]\label{lem:ect_tc}
$ \ECT \equiv_\W (\TC_\Nat)^*.$
\end{lemma}

The following shows that two of the benchmark problems above, namely $(\LPO')^*$ and $(\TC_\Nat)^*$
are Weihrauch-incomparable. It is established as \cite[Lemma 9]{Pau-Pra-Sol:24}.

\begin{lemma}[Pauly, Pradic, Soldà]
$(\LPO')^*$ and $(\TC_\Nat)^*$ are Weihrauch incomparable.
\end{lemma}

Note that the previous lemma implies that $(\LPO')^* <_\W (\LPO')^* \times
(\TC_\Nat)^*$ and $(\TC_\Nat)^* <_\W (\LPO')^* \times (\TC_\Nat)^*$, since $\PP, \Q \leq_\W \PP \times \Q$.

Our main result is the following.

\begin{theorem}\label{thm:main}
$\ORT_\Nat \equiv_\W \ECT \times \WOP(X\mapsto X^\omega)$.
\end{theorem}

To complete the picture we furthermore establish the following points.

\begin{itemize}
\item $\WOP(X\mapsto X^\omega)$ and $(\LPO')^*$ are incomparable, and
\item $\ECT$ is strictly weaker than $\WOP(X\mapsto X^\omega)$, and 
\item $\WOP(X\mapsto X^\omega)$ is strictly weaker than $\ECT\times \WOP(X\mapsto X^\omega)$. 
\end{itemize}

In Section \ref{sec:equivalence} we prove the Weihrauch equivalence of the Ordered Ramsey Theorem 
with the parallel product of the Eventually Constant Tail principle and the WOP for $X \mapsto X^\omega$. 
In Section \ref{sec:auxiliary} we complete the picture by proving that the WOP for $X \mapsto X^\omega$ is 
Weihrauch-incomparable with $(\LPO')^*$. This shows that our result is give a different splitting of $\ORT_\Nat$
than the one from Theorem \ref{thm:pauly} above. We further show that $\ECT$ is strictly weaker than 
$\WOP(X\mapsto X^\omega)$ but indispensable in the characterization of $\ORT_\Nat$.

\section{The Ordered Ramsey Theorem and $\WOP(X\mapsto X^\omega)$}\label{sec:equivalence}

In this section we establish the Weihrauch equivalence of the Ordered Ramsey Theorem with the parallel product of 
the Eventually Constant Tail principle and the WOP for the operator $X \mapsto X^\omega$. The lower bound on the Ordered Ramsey Theorem is based on an idea from \cite{Car-Mai-Zda:25}. The upper bound is, as far as we know, the first example of a reduction of a Ramsey-type principle to a WOP principle.

\subsection{Lower bound: $\ORT_{\mathbb N}\geq_W \ECT \times \WOP(X\mapsto X^\omega)$}

Let $X$ be a linearly ordered set. Let $\tau =\langle (b_0,a_0),\dots, (b_{m}, a_{m}) \rangle$ be in $X^\omega$.
Define the {\em length} of $\tau$ as $|\tau|:=2m+2$. For each $0\leq s < |\tau|$, define the $s$-th {\em component} of $\tau$ as
    $$
    \tau(s):=
    \begin{cases}
        b_{t} \in \mathbb N, & \text{if } s=2t \text{ is even};\\
        a_{t} \in X, & \text{if } s=2t+1 \text{ is odd.}
    \end{cases}
    $$
Then $|\tau|$ is the least $s\geq 0$ such that $\tau(s)$ is undefined.\\
We use the above terminology to add an order structure to the coloring devised in \cite{Car-Mai-Zda:25} to establish a reduction from $\RT^2_\infty$ to $\WOP(X\mapsto X^\omega)$. We use $\prec$ to denote the prefix relation on sequences.
 
\begin{proposition}
\label{wop_ort}
    $\WOP(X\mapsto X^\omega) \leq_{\W} \ORT_\mathbb N$.
    \begin{proof}
    Let $\sigma=(\sigma_i)_{i\geq 0}$ be an instance of $\WOP(X\mapsto X^\omega)$. If
    \[\sigma_i=\langle (b_0^i,a_0^i),\dots, (b_{m_i}^i, a_{m_i}^i) \rangle,\]
    then $|\sigma_i|=2m_i+2$ and $m_i\leq b_0^i\leq b_0^0=:k_\sigma$. Hence $|\sigma_i|\leq 2k_\sigma+2$ for every $i\geq 0$.\\   
    Define a coloring $c:[\Nat]^2 \to 2k_\sigma+2$ as follows
  $$
    c(i,j):=
    \begin{cases}
        |\sigma_j|  & \text{if } \sigma_j \prec \sigma_i;\\
        \text{least } s < |\sigma_i| \text{ such that } \sigma_i(s)\neq \sigma_j(s) & otherwise.
 \end{cases}
    $$
If $\sigma_j \nprec \sigma_i$, since $\sigma_j <_{X^\omega} \sigma_i$ there exists $0\leq s<|\sigma_i|$ such that $\sigma_i(s)\neq \sigma_j(s)$. Also, in both cases of the definition of $c$, $c(i,j)<|\sigma_i|$ and $\sigma_j(s)=\sigma_i(s)$ for all $0\leq s<c(i,j)$.\\    
    Notice that for all $0\leq i<j<k$ we have $c(i,k)\leq c(i,j)$. Suppose otherwise, as witnessed by $0\leq i<j<k$ such that $c(i,j)<c(i,k)$. Then $\sigma_k(s)=\sigma_i(s)$ for all $0\leq s\leq c(i,j)$. On the other hand, $\sigma_j(s)=\sigma_i(s)$ for all $0\leq s<c(i,j)$ and $\sigma_j(c(i,j))$ is either undefined or strictly less than $\sigma_i(c(i,j))$, whether in $\mathbb N$ or in $X$, depending on the parity of $c(i,j)$. It follows that $\sigma_k<_{X^\omega }\sigma_j $, which is impossible.\\
    Let $P:=[2k_\sigma+2]$, and define $\leq_P$ as the reverse of the natural order,
    \[ 0>_P 1>_P\dots \:>_P2k_\sigma+1.\]
    Then the above remarks show that $c:[\mathbb N]^2\to P$ and that $c$ is right-ordered with respect to $\leq_P$. Hence $(P,c)$ is a $\sigma$-computable instance of $\ORT_{\mathbb N}$.\\       
    Let $H=\{h_0< h_1< \dots\}$ be a solution to $\ORT_\mathbb N$ for $(P,c)$. Denote by $c^*$ the $c$-color of $[H]^2$. Notice that $c^*=c(h_i,h_{i+1})<|\sigma_{h_i}|$ for all $i\geq 0$, hence $\sigma_{h_i}(c^*)$ is defined for all $i\geq 0$.\\       
    Now, suppose by way of contradiction that $c^*$ is even. Then    
    \[b_{\frac{c^*}{2}}^{h_0}>_\mathbb N\:b_{\frac{c^*}{2}}^{h_1} >_\mathbb N \dots\]
    which is impossible.\\
    Therefore $c^*$ is odd, and 
    \[a_{\frac{c^*-1}{2}}^{h_0}>_X\:a_{\frac{c^*-1}{2}}^{h_1} >_X\dots\]
    is the desired $(H\oplus \sigma)$-computable solution to $\WOP(X\mapsto X^\omega)$ for $\sigma$.
    \end{proof}
\end{proposition}

\begin{lemma} 
\label{ort_star}
$(\ORT_\mathbb N)^* \equiv_{\sW} \ORT_\mathbb N$.
    \begin{proof}
     Obviously $\ORT_\mathbb N \leq_{\sW} (\ORT_\mathbb N)^*$. We prove the opposite inequality. Let $\big \langle n, (P_1, c_1),\dots, (P_n, c_n)\big \rangle$ be an instance of $(\ORT_\mathbb N)^*$. Then, for all $1\leq i\leq n$, each $c_i: [\mathbb N]^2\to P_i$ is right-ordered with respect to $\leq_{P_i}$.\\
     Let $P:=P_1\times\dots\times P_n$, and define a coloring $c: [\mathbb N]^2\to P$ by
     \[ c(x,y):=\big( c_1(x,y), \dots, c_n(x,y)\big).\]
    We denote by $\leq_P$ the product order on $P$, i.e., for $1\leq i\leq n$
    \[(p_1,\dots, p_n)\leq_P(q_1,\dots, q_n)\;\mbox{ if and only if}\; p_i\leq_{P_i}q_i.\]
    Suppose that $x<y<y'$. Then $c_i(x,y)\leq_{P_i} c_i(x,y')$ for all $1\leq i\leq n$, hence
    \[ c(x,y)=\big(c_1(x,y),\dots, c_n(x,y)\big)\leq_P \big(c_1(x,y'),\dots, c_n(x,y')\big).\]
    This shows that $c$ is right-ordered with respect to $\leq_P$, i.e.~that $(P,c)$ is 
    an instance of $\ORT_\mathbb N$. \\    
    Let $H$ be a solution to $\ORT_\mathbb N$ for $(P,c)$. If $p=(p_1,\dots, p_n)\in P$ is the $c$-color of $[H]^2$ then $p_i\in P_i$ is the $c_i$-color of $[H]^2$ for each $1\leq i\leq n$. Therefore $\langle H,\dots, H\rangle$ is a solution to $(\ORT_\mathbb N)^*$ for $\big \langle n, (P_1,c_1),\dots, (P_n, c_n)\big \rangle$.
    \end{proof}
\end{lemma}

\begin{proposition}
$\ECT\times \WOP(X\mapsto X^\omega) \leq_{\W} \ORT_\mathbb N$.
    \begin{proof}
 	From \cite[Lemma 52]{Pau-Pra-Sol:24} it follows that $\ECT \leq_{\sW} \ORT_\mathbb N$.
         From this and Proposition \ref{wop_ort} we have $\ECT\times \WOP(X\mapsto X^\omega) \leq_{\W} (\ORT_\mathbb N)^2$. Clearly $(\ORT_\mathbb N)^2\leq_{\sW} (\ORT_\mathbb N)^*$ and,  by Lemma \ref{ort_star},  $(\ORT_\mathbb N)^*\leq_{\sW} \ORT_\mathbb N$. 
    \end{proof}
\end{proposition}

\subsection{Upper bound: $\ORT_{\mathbb N}\leq_\W \ECT \times \WOP(X\mapsto X^\omega)$}

In this subsection we prove that $\ORT_{\mathbb N}\leq_\W \ECT \times \WOP(X\mapsto X^\omega)$. We obtain 
the result from a fine-grained parametrized version. Let $n\geq 1$. Let $\ORT_\mathbb N[n]$ denote the restriction of $\ORT_{\mathbb N}$ to instances $(P, c)$ with $|P|=n$. Let $\WOP(X\mapsto X_{lex}^n)$ denote the well-ordering principle regarding the lexicographic order $<_{lex}$ on $X^n$: an instance of $WOP(X\mapsto X_{lex}^n)$ is a $<_{lex}$-descending sequence $\sigma=(\sigma_i)_{i\geq 0}$ in $X^n$ for some linearly ordered set $X$, and a corresponding solution is a descending sequence in $X$ whose terms occur as components of (terms of) $\sigma$.

\begin{proposition}
\label{ort[n]} For all $n\geq 1$, 
    $\ORT_\mathbb N[n] \leq_{\sW} (\ECT)^n\times \WOP(X\mapsto X_{lex}^n)$.
    \begin{proof}
        Let $(P,c)$ be an instance of $\ORT_\mathbb N[n]$. This means that $P$ is a finite poset of cardinality $n$, and $c:[\mathbb N]^2\to P$ is right-ordered with respect to $\leq_P$.\\
        We begin by fixing a linear ordering $p_1<p_2<\dots <p_n$ of $P$ extending $\leq_P$. This can be done computably in $P$.\\
        For each $1\leq j\leq n$ we define a binary sequence $improvement_j$ and an auxiliary sequence $previous_j$ as follows. First, set       
        \[improvement_j(0):=1\;,\;\;previous_j(0)=previous_j(1):=0.\]
        Now let $i\geq 1$. If there exists $previous_j(i)\leq k<i$ such that $c(k,i)=p_j$, then call $k_j^i$ the least such $k$ and set $improvement_j(i):=1$, $previous_j(i+1):=i$. Otherwise, set  $improvement_j(i):=0$, $previous_j(i+1)=previous_j(i)$. We say that $p_j$ improves at stage $i\geq 0$ if and only if $improvement_j(i)=1$. Notice that both $improvement_j$ and $previous_j$ are $(P\oplus c)$-computable.\\
        When interpreted as $2$-colorings of $\mathbb N$, the sequences $improvement_j$, $1\leq j\leq n$, constitute an instance of $(\ECT)^n$. Let $n_0^{j}$, $1\leq j\leq n$, be the corresponding solutions, and set $n_0:=\max_{1\leq j\leq n}n_0^j$. Then for all $1\leq j\leq n$ either $p_j$ doesn't improve after stage $n_0$ or $p_j$ improves infintely many times after stage $n_0$.\\
        Next, we define the fitting $(P\oplus c)$-computable instance $\sigma=(\sigma_i)_{i\geq0}$ of $\WOP(X\mapsto X_{lex}^n)$.
        We choose $X:= \{1, \dots, n\}\times \mathbb N \times \mathbb Q \times \{-n,\dots,-1\}$. Each factor is equipped with its natural order, so that the resulting lexicographic order on $X$ is linear.\\    
        Stage $0$. Set $\sigma_0:=\big(a_n^0,\dots, a_{1}^0\big)$, where $a_j^0:=\big(j,\:1,\:2,-1\big)$ for all $n\geq j\geq 1$.\\      
        Stage $i\geq 1$. Consider the set $I^i$ of all $j$-s such that $p_j$ improves at stage $i$. For every $j\in I^i$, let $n\geq m^i_j\geq j$ be the largest $m$ such that $p_m$ has improved at least once in the timeframe $(previous_j(i), i]$.\\
        Suppose that $j\in I^i$. If $m^i_j=j$, set
        \[a_j^i:=\big(\:j,\: a_j^{i-1}(2),\: \frac{1}{k_j^i+1},\: -1\:);\]
        if, otherwise,  $m_j^i> j$, set
        \[a_j^i:=\big( \:j,\: a_j^{i-1}(2)+1,\: \frac{1}{k_j^i+1},\:-1 \big).\]
        Suppose instead that $j\notin I^i$. If $j=\max \{m_l^i\;|\;l\in I^i\}$ then set
        \[a_j^i:=\big(j,\: a_j^{i-1}(2),\:a_j^{i-1}(3),\: a_j^{i-1}(4)-1\big);\]
        otherwise, set
        \[a_j^i:=a_j^{i-1}.\]
        Finally, set $\sigma_i:=\big(a_n^i,\dots, a_{1}^i\big)$.\\
        Notice that $I^i$ is non-empty for every $i\geq 1$, as it contains the index of $c(i-1,i)$. Then it's easy to see that $\sigma_{i}<_{lex} \sigma_{i-1}$, as required. This completes the definition of $\sigma$.\\
        Let $\overrightarrow x= (x_s)_{s\geq 0}$ be a sequence in $X$ that is a solution to $\WOP(X\mapsto X_{lex}^n)$ for $\sigma$. We aim to compute an increasing enumeration $\overrightarrow y=(y_s)_{s\geq 0}$ of natural numbers for a $c$-homogenous set from $n_0$ and $\overrightarrow x$. The $x_s$-s are drawn from the matrix 
        \[\Sigma:=(a_j^i)_{n\geq j\geq 1}^{i\geq 0},\]
        where $i$ is interpreted as a row index and $j$ as a (decreasing) column index. We are able to locate the $x_s$-s more precisely by reasoning as follows. Let $j^*$ be the largest $n\geq j\geq 1$ such that $p_{j}$ improves infinitely many times.  Let $i^*$ be the least $i\geq 0$ such that $p_{j*}$ improves at stage $i$ and no color $p_j$, $n\geq j>j^*$, improves after stage $i$. Then:

        \begin{enumerate}

            \item \label{uno}
            If $p_j$ improves finitely many times, then 
            $a_{j}^i(1)$, $a_{j}^i(2)$, and $a_{j}^i(3)$ are constant for $i\geq n_0$; also, by construction $a_{j}^i(4)>_Xa_{j}^{i+1}(4)$ for at most $(n-1)$ many  $i\geq n_0$. Hence column $j$ contains at most $(n_0+1)+(n-1)=n_0+n$ distinct coefficients.           

            \item \label{due}
            Together, columns $n\geq j \geq j^*+1$ contain at most $M:= n(n_0 +n)$ distinct coefficients. This follows immediately from remark (\ref{uno}) and the choice of $j^*$.
            
            \item \label{tre}
            $\overrightarrow x$ can't move from a given column to a column on its left. This is true because by construction $a_j^i>_Xa_{j'}^{i'}$ for all $n\geq j>j'\geq 1$ and  for all $i,i'\geq 0$.
            
            \item \label{quattro}
            If $j^*-1\geq j\geq 1$, then for every $i\geq 0$ column $j$ contains finitely many distinct coefficients smaller than $a_j^i$. This follows trivially from (\ref{uno}) if $p_j$ doesn't improve indefinitely. Suppose otherwise: since $p_{j^*}$ improves at some stage $i'>i$ and $j^*>j$, $a_j^i(2)$ is eventually increased.

            \item \label{cinque}
            $\overrightarrow x$ can't visit columns $j^*-1\geq j\geq 1$. This follows from remarks (\ref{tre}) and (\ref{quattro}) combined.

            \item \label{sei}
            $x_{s}$ lies in column $j^*$ for all $s\geq M+1$.
            This follows from remarks (\ref{due}) and (\ref{cinque}) combined.

            \item \label{sette}
            Let $z:=a_{j^*}^{i^*}(2)$. Then $a_{j^*}^{i}(2)< z$ for all $0\leq i<i^*$ and $a_{j^*}^{i}(2)=z$ for all $i\geq i^*$. Hence if $i<i^*$ then $\sigma$ contains finitely many distinct coefficients smaller than $a_{j^*}^{i}$, meaning that $\overrightarrow x$ can't visit the square $a_{j^*}^{i}$.

            \item \label {otto}
            For $s\geq M+1$, the $x_s$-s are drawn in increasing order of rows from column $j^*$ (so that $x_{s+n}(3)<x_{s}(3)<2$ for all $s\geq M+2$.). This follows from remarks (\ref{sei}) and (\ref{sette}), since $x_s(2)=z$ is fixed for $s\geq M+1$. 
            
        \end{enumerate}
        Set
        \[y_s:=\frac{1}{x_{f(s)}(3)}-1\;\;\;\;\forall s\geq 0,\] 
        where $f(s):=(M+2)+(s\times n)$. Since $f$ is computable from $n_0$, $\overrightarrow y$ is computable from $(n_0\oplus \overrightarrow x)$. Remark (\ref{otto}) above proves that $\overrightarrow y$ is an increasing sequence of natural numbers.\\
        It only remains to show that $\overrightarrow y$ is $c$-monochromatic. We fix $0\leq s< t$, and show that $c(y_s,y_t)=p_{j^*}$. Let $i$ be least such that $a_{j^*}^{i}(3)=x_{f(s)}(3)$, so that $y_s=k_{j^*}^{i}$ and $c(k_{j^*}^{i},i)=p_{j^*}$. Similarly, let $i'$ be least such that $a_{j^*}^{i'}(3)=x_{f(t)}(3)$, so that $y_t=k_{j^*}^{i'}$. By construction we have 
        $i^*\leq k_{j^*}^{i}<i\leq k_{j^*}^{i'}$. Since $c$ is right-ordered, $c(k_{j^*}^{i},k_{j^*}^{i'})\geq_P p_{j^*}$. Now if $c(k_{j^*}^{i},k_{j^*}^{i'})=p_j>_P p_{j^*}$ then $j>j^*$ and $p_j$ improves after $i^*$, contradicting the choices of $i^*$ and $j^*$. Therefore $c(y_s,y_t)= c(k_{j^*}^{i},k_{j^*}^{i'})= p_{j^*}$, as required.        
    \end{proof}
\end{proposition}

Let $\WOP(X\mapsto X^\omega)[\leq n]$ denote the restriction of $\WOP(X\mapsto X^\omega)$ to those sequences $\tau=(\tau_i)_{i\geq 0}$ satisfying $|\tau_i|\leq n$ for all $i\geq 0$.

\begin{lemma} \label{wop_lex}
For all $n\geq 1$, $\WOP(X\mapsto X_{lex}^n)\leq_{\sW} \WOP(X\mapsto X^\omega)[\leq 2n]$
\begin{proof}
    Let $\sigma =(\sigma_i)_{i\geq 0}$ be an instance of $\WOP(X\mapsto X_{lex}^n)$, say $\sigma_i=\big(a_n^i,\dots, a_{1}^i\big)$.
    Set $\tau_i:=\langle(n, a_n^i), (n-1, a_{n-1}^i),\dots,(1, a_{1}^i)\rangle$ for all $i\geq 0$. Then $\tau_{i+1}<_{X^\omega}\tau_{i}$ and $|\tau_i|=2n$. Hence $\tau$ is a $\sigma$-computable instance of $ \WOP(X\mapsto X^\omega)[\leq 2n]$. It's clear that every solution to $ \WOP(X\mapsto X^\omega)[\leq 2n]$ for $\tau$ is also a solution to $\WOP(X\mapsto X_{lex}^n)$ for $\sigma$.
\end{proof}
\end{lemma}

\begin{corollary}
    $\ORT_{\mathbb N}\leq_\W \ECT \times \WOP(X\mapsto X^\omega)$
\begin{proof}
    By proposition \ref{ort[n]} and lemma \ref{wop_lex} we have $\ORT_{\mathbb N}[n]\leq_{\sW} (\ECT)^n \times \WOP(X\mapsto X^\omega)[\leq 2n]$. Hence
    $\ORT_{\mathbb N}\leq_\W (\ECT)^* \times \WOP(X\mapsto X^\omega)$ (the argument in proposition \ref{due} doesn't establish a strong Weihrauch reduction, since $f$ depends on $n=|P|$). 
From Lemma \ref{lem:ect_tc} it follows that $(\ECT)^*\equiv_\W \ECT$, since $(\PP^*)^* \equiv_\W \PP^*$ for every problem $\PP$. Therefore the desired inequality $\ORT_{\mathbb N}\leq_\W \ECT \times \WOP(X\mapsto X^\omega)$ holds.
\end{proof}
\end{corollary}\textbf{}

\section{Completing the picture}\label{sec:auxiliary}

In the previous section we proved the equivalence $\ORT_\Nat \equiv_\W \ECT \times \WOP(X \mapsto X^\omega)$. Since $\ORT_\Nat$ is also equivalent to $\ECT \times (\LPO')^*$ by Theorem \ref{thm:pauly} from \cite{Pau-Pra-Sol:24}, it is natural to ask how $\WOP(X \mapsto X^\omega)$ compares with $(\LPO')^*$ and with $\ECT$. 
We show that $\WOP(X\mapsto X^\omega)$ and $(\LPO')^*$ are incomparable and that $\ECT$ is strictly 
weaker than $\WOP(X\mapsto X^\omega)$. 

\subsection{$\WOP(X\mapsto X^\omega)$  and $(\LPO')^*$ are Weihrauch incomparable}

The proof of the following proposition is a variation on the proof of \cite[Lemma 8]{Pau-Pra-Sol:24} which shows $\IsFinite_\mathbb S \nleq_\W \ECT$.

\begin{proposition}
\label{isfinite_wop}
    $\IsFinite_\mathbb S \nleq_\W \WOP(X\mapsto X^\omega)$.
\begin{proof}    Suppose by way of contradiction that $\IsFinite_\mathbb S \leq_\W \WOP(X\mapsto X^\omega)$, as witnessed by the computable Turing functionals $K$ and $H$. This means that for every instance $p$ of $\IsFinite_\mathbb S$, $H(p)$ is an instance of $\WOP(X\mapsto X^\omega)$, and for every solution $\overrightarrow  x$ to $\WOP(X\mapsto X^\omega)$ for $H(p)$, $K \langle\overrightarrow  x, p\rangle=\IsFinite_\mathbb S(p)$.\\
    Consider $s^0:=0^\mathbb N$ and any solution $\langle y_0^0, y_1^0,\dots\rangle$ to $\WOP(X\mapsto X^\omega)$ for $H(s^0)$. Since $\IsFinite_\mathbb S(s^0)=\top$,
    \[K\langle\langle y_0^0,\dots, y_{m_0}^0\rangle, 0^{n_0}\rangle \downarrow \;= \top\]
    for some $m_0,n_0\geq 0$. By choosing a large enough $n_0$, we can make sure that all the terms $y_0^0,\dots,y_{m_0}^0$ are mentioned in $H(0^{n_0})$.\\
    \comment{Next, consider $s^1:=0^{n_0}10^\mathbb N$ and any solution $\langle y_0^1, y_1^1, \dots\rangle$ to $\WOP(X\mapsto X^\omega)$ for $H(s^1)$. Since $\IsFinite(s^1)=\top$, 
    \[K\langle\langle y_0^1, \dots, y_{m_1}^1\rangle, 0^{n_0}10^{n_1}\rangle = \top\]
    for some $m_1, n_1\geq 0$. Again, we can assume that $H(0^{n_0}10^{n_1})$ mentions the terms $y_0^1,\dots,y_{m_1}^1$.\\}
    We proceed inductively to obtain, for each $k\geq 1$: a binary sequence $s^k:=0^{n_0}10^{n_1} \dots 10^{n_{k-1}}10^\mathbb N$; a solution $\langle y_0^k, y_1^k,\dots\:\rangle$ to $\WOP(X\mapsto X^\omega)$ for $H(s^k)$; two integers $m_k, n_k\geq 0$ such that
    \[K\langle\langle y_0^k, \dots, y_{m_k}^k\rangle, 0^{n_0}10 \dots 10^{n_{k-1}}10^{n_k}\rangle \downarrow \; = \top\]
    and such that $y_0^k,\dots,y_{m_k}^k$ are all mentioned in $H(0^{n_0}10 \dots 10^{n_{k-1}}10^{n_k})$.\\
    Now, let
    \[p:=0^{n_0}10^{n_{1}}10^{n_2}10^{n_3}1\dots\:,\]
    and let $\langle x_0, x_1,\dots \rangle$ be a solution to $\WOP(X\mapsto X^\omega)$ for $H(p)$. Since $\IsFinite_\mathbb S(p)=\bot$, 
    \[K\langle\langle x_0, \dots, x_{m}\rangle, 0^{n_0}1 \dots 10^{n_{k}}\rangle \downarrow \; = \bot\]
    for some $m, k\geq 0$, and again we can assume without loss of generality that $x_0,\dots, x_m$ are all mentioned in $H(0^{n_0}1 \dots 10^{n_{k}})$. The latter, which is an initial segment of both $H(s^k)$ and $H(p)$, also contains the terms $ y_0^k, \dots, y_{m_k}^k$, and thus decides on the relation between $x_m$ and $y_{m_k}^k$ in the linear ordering $\leq_X$.\\
    If $x_m\geq_X y_{m_k}^k$, then $\overrightarrow z:=\langle x_0,\dots, x_m, y_{m_{k+1}}^k, y_{m_{k+2}}^k,\: \dots \rangle$ is a solution to $\WOP(X\mapsto X^\omega)$ for $H(s^k)$. Hence 
    \[K\langle \overrightarrow{z} , s^k\rangle = \top. \]
    But 
    \[K\langle \overrightarrow{z} , s^k\rangle= K\langle \langle x_0, \dots, x_m\rangle  , 0^{n_0}1 \dots 10^{n_{k}}\rangle=\bot,\]
    which is impossible.\\
    Similarly, if $x_m\leq_X  y_{m_k}^k$ then $\overrightarrow w:=\langle y_0^k,\dots, y_{m_k}^k, x_{m+1}, x_{m+2},\: \dots\rangle$ is a solution to $\WOP(X\mapsto X^\omega)$ for $H(p)$, hence 
    \[K\langle \overrightarrow{w} , p\rangle = \bot. \]
    But 
    \[K\langle \overrightarrow{w} , p\rangle= K\langle \langle y_0^k, \dots, y_{m_k}^k\rangle  , 0^{n_0}1 \dots 10^{n_{k}}\rangle=\top,\]
    which is again impossible. This gives the desired contradiction.\\
    \end{proof}
\end{proposition}

\begin{corollary}
\label{lpo'_wop}
    $(\LPO')^*\nleq _\W \WOP(X\mapsto X^\omega)$. 
    \begin{proof}
         Combine Proposition \ref{isfinite_wop} with the fact that $\IsFinite_\mathbb S\leq_\W \LPO'$ (as established in the proof of \cite{Pau-Pra-Sol:24}, Lemma 9).
    \end{proof}
\end{corollary}
    
\begin{corollary}
   $\WOP(X \mapsto X^\omega) \lneq_\W \ECT\times \WOP(X \mapsto X^\omega)$.
   \begin{proof}
       We already know that $\ECT\times \WOP(X \mapsto X^\omega)\equiv_\W \ORT_\mathbb N\equiv_\W \ECT\times (\LPO')^*$. Then the assumption $\ECT\times \WOP(X \mapsto X^\omega)\leq_\W \WOP(X \mapsto X^\omega)$ implies $(\LPO')^*\leq _\W \WOP(X\mapsto X^\omega)$,  which contradicts corollary \ref{lpo'_wop}.
   \end{proof}
\end{corollary}

The proof of the following proposition is a variation on the proof of \cite[Lemma 9]{Pau-Pra-Sol:24}.

\begin{proposition}
\label{lex_lpo'}
    $\WOP(X\mapsto X_{lex}^2)\nleq_\W (\LPO')^*$.
    \begin{proof}        Suppose by way of contradiction that $\WOP(X\mapsto X_{lex}^2)\leq_\W (\LPO')^*$, as witnessed by the computable Turing functionals $K$ and $H$. This means that: for every instance $\sigma$ of $\WOP(X\mapsto X_{lex}^2)$, $H(\sigma)$ is an instance of $(\LPO')^*$, say $H(\sigma)=\langle n_{\sigma}; \rho^\sigma_0,\dots, \rho^\sigma_{n_\sigma-1} \rangle$; if $\eta\in 2^{n_\sigma}$ is a solution to $(\LPO')^*$ for $H(\sigma)$, then $K\langle \eta,\: \sigma \rangle$ is a solution to $\WOP(X\mapsto X_{lex}^2)$ for $\sigma$. We aim to build an instance $\sigma$ infinite decreasing in $X^2_{lex} := \mathbb{Q}^2_{lex}$ for which the above fails.\\
        Start by letting $\tau^0:=\langle \tau_0^0,\:\tau_1^0,\dots \rangle$, $\tau_i^0:= \langle 0, \frac{1}{i+1}\rangle$ for all $i\geq 0$. Fix some $\bar{i}\geq 0$ such that the first component of $H$ converges to $n_{\tau^0}=:n$ after having read the rows $\tau_0^0,\dots, \tau_{\bar i}^0$ of $\tau^0$. Then for every instance $\tau=\langle \tau_0,\tau_1,\dots \rangle$ of $\WOP(X\mapsto X_{lex}^2)$ such that $\tau_i=\tau_i^0$ for all $0\leq i\leq \bar i$ we will also have $n_\tau=n$. Note a solution to $(\LPO')^*$ for $H(\tau^0)$ is a binary strings of length $n$.\\
        Let $\eta^0,\dots,  \eta^{2^n-1}$ be an enumeration of $2^n$. Let $\eta^{k_0}$ be the solution to $(\LPO')^*$ for $H(\tau^0)$. Since the solutions to $\WOP(X\mapsto X_{lex}^2)$ for $\tau^0$ are exactly the infinite subsequences of its second column, i.e. of the form $(\frac{1}{m^0_\ell})_{\ell \in \Nat}$ where $1 \leq m^0_0 < m^0_1 < m^0_2 <  \dots$, there exists $i_0\geq \bar i$ such that
        \[K\langle \eta^{k_0}, \langle \tau_0^0,\:...\:,\tau_{i_0}^0\rangle \rangle \downarrow \;=\langle \frac{1}{m_0^0},\dots, \frac{1}{m_{h_0}^0} \rangle  \]
        for some integers $h_0\geq 0$, $1\leq m_{0}^0<\dots\:<m_{h_0}^0$.\\
        Let $1\leq j < 2^{n}$. 
        Suppose inductively that an instance $\tau^{j-1}=\langle \tau_0^{j-1},\:\tau_1^{j-1},\dots \rangle$ of $\WOP(X\mapsto X_{lex}^2)$ and an integer $i_{j-1}\geq \bar i$ have been defined so as to satisfy:
        \begin{equation} \label{ind}
            \tau_i^{j-1}(l)\leq j\;\; \text{ for all }0\leq i\leq i_{j-1},\: l\in \{0,1\} .
        \end{equation}
        Set $\tau^j:=\langle \tau_0^j,\tau_1^j,\dots \rangle$,  where $\tau_i^j:=\tau_i^{{j-1}}$ for $0\leq i\leq i_{j-1}$ and $\tau_{i_{j-1}+i}^j:=\langle -j, j+\frac{1}{i}\rangle $ for $i\geq 1$. It follows from (\ref{ind}) that the solutions to $\WOP(X\mapsto X_{lex}^2)$ for $\tau^j$ are exactly the infinite subsequences of $\langle j + \frac{1}{m}: m\geq 1\rangle$. Hence if $\eta^{k_j}$ is the solution to $(\LPO')^*$ for $H(\tau^j)$, there exists $i_j>i_{j-1}$ such that
        \[K\langle \eta^{k_j}, \langle \tau_0^j,\dots,\tau_{i_j}^j\rangle \rangle \downarrow \;=\langle\: j+\frac{1}{m_0^j},\dots,\: j+\frac{1}{m_{h_j}^j} \:\rangle\]
        for some integers $h_j\geq 0$, $1\leq m_{0}^j<\dots \:<m_{h_j}^j$.\\
        Finally, set $\sigma:=\langle\sigma_0,\sigma_1,\dots \rangle$, where $\sigma_i:= \tau_i^{2^n-1}$ for $0\leq i\leq i_{2^n-1}$ and $\sigma_{i_{2^n-1}+i}:=\langle-2^n, 2^n+\frac{1}{i} \rangle $ for all $i\geq 1$. Notice that the indexes $k_0,\dots, k_{2^n-1}$ chosen above are necessarily all distinct. Hence there is a unique $0\leq j<2^n$ such that $\eta^{k_j}$ is the solution to $(\LPO')^*$ for $H(\sigma)$. By construction,
        $K$ outputs $\langle\: j+\frac{1}{m_0^j},\dots, \: j+\frac{1}{m_{h_j}^j} \:\rangle$ after having read the string $\eta^{k_j}$ together with the rows $\sigma_0=\tau_0^j,\dots, \sigma_{i_j}=\tau_{i_j}^j$ of $\sigma$. This contradicts the assumption on $K$, because the solutions to $\WOP(X\mapsto X_{lex}^2)$ for $\sigma$ are exactly the subequences of $\langle 2^n+ \frac{1}{m} : m\geq 1\rangle $.\\
    \end{proof}
\end{proposition}

\begin{corollary}
    $\WOP(X\mapsto X^\omega)\nleq_\W (\LPO')^*$.
    \begin{proof}
        This follows immediately from proposition \ref{lex_lpo'}, by recalling from Lemma \ref{wop_lex} that $\WOP(X\mapsto X_{lex}^n)\leq_{\sW} \WOP(X\mapsto X^\omega)[2n]$ for every $n\geq 1$.
    \end{proof}
\end{corollary}

In \cite{Pau-Pra-Sol:24} it is shown that $(\LPO')^*$ is Weihrauch-equivalent to the variants of $\ORT_\Nat$ and $\ART_\Nat$
in which the output is the color of some infinite homogeneous set for the instance coloring (these variants are denoted $c\ORT_\Nat$
and $c\ART_\Nat$ in \cite{Pau-Pra-Sol:24}. By the results of the present section we have the obvious corollary that $\WOP(X\mapsto X^\omega)$ is incomparable with these principles as well. 

\subsection{$\ECT$ is strictly weaker than $\WOP(X\mapsto X^\omega)$}

We first establish that $\WOP(X \mapsto X^\omega)$ is not reducible to $\ECT$. 
The proof of the following proposition is a variation on the proof of \cite[Lemma 9]{Pau-Pra-Sol:24}.

\begin{proposition}
\label{wop_ect}
    $\WOP(X\mapsto X_{lex}^2) \nleq_\W \ECT$.
    \begin{proof}
        Suppose by way of contradiction that $\WOP(X\mapsto X_{lex}^2)\leq_\W \ECT$, as witnessed by the computable Turing functionals $K$ and $H$. This means that for every instance $\sigma$ of $\WOP(X\mapsto X_{lex}^2)$, $H(\sigma):\mathbb N\to n_\sigma$ is an instance of $\ECT$, and for every bound $b$ for $H(\sigma)$, $K\langle b,\: \sigma \rangle$ is a solution to $\WOP(X\mapsto X_{lex}^2)$ for $\sigma$. We build a sequence $\sigma$ in $(2\times_{lex}\mathbb Q)^2$ for which the above fails.\\
        Start by letting $\tau^0:=\langle \tau_0^0,\:\tau_1^0,\dots \rangle$, $\tau_i^0:= \langle (1,1), (0,\frac{1}{i+1})\rangle$ for all $i\geq 0$. Then a solution to $\WOP(X\mapsto X_{lex}^2)$ for $\tau^0$ consists of a (possibly empty) subsequence of $\langle (1,1)\rangle$ followed by an infinite subsequence of $\langle (0, \frac{1}{m}):m\geq 1 \rangle$. Pick any bound $b_0$ for $H(\tau^0)$: by the assumption on $K$, there exists $i_0\geq 0$ such that
        \[K\langle b_0, \langle \tau_0^0,\dots, \tau_{i_0}^0\rangle \rangle \downarrow \;=\langle\alpha^0, (0,\frac{1}{m_0^0}),\dots, (0,\frac{1}{m_{h_0}^0}) \rangle  \]
        for some $\alpha^0\subseteq \langle (1,1)
        \rangle$ and some integers $h_0\geq 0$, $1\leq m_{0}^0<\dots\:<m_{h_0}^0$.\\
        Let $j\geq 1$. 
        Suppose inductively that the instances $\tau^l=\langle \tau_0^l,\:\tau_1^l,\dots \rangle$, $0\leq l<j$, of $\WOP(X\mapsto X_{lex}^2)$ and the integers $0<i_0< \dots <i_{j-1}$ have been defined so as to satisfy:
        \begin{equation} \label{indu}
            \text{for all $1\leq l<j$},\; \tau_i^l=\tau_i^{l-1} \;\text{ if }\; 0\leq i\leq i_{l-1} \;\text{ and }\; \tau_{i_{l-1}+i}^l=\langle(1, \frac{1}{l+1}),(0,l+\frac{1}{i}) \rangle \;\text{ for }\; i\geq 1.
        \end{equation}
        In order to preserve the inductive assumption, we must define $\tau^j$ as $\langle \tau_0^j,\tau_1^j,\dots \rangle$,  with $\tau_i^j:=\tau_i^{{j-1}}$ for $0\leq i\leq i_{j-1}$ and $\tau_{i_{j-1}+i}^j:=\langle (1,\frac{1}{j+1}), (0,j+\frac{1}{i})\rangle $ for $i\geq 1$. Then it follows from (\ref{indu}) that a solution to $\WOP(X\mapsto X_{lex}^2)$ for $\tau^j$ consists of a (possibly empty) subsequence of $\langle (1,1),\dots,(1,\frac{1}{j+1})\rangle$ followed by an infinite subsequence of $\langle (0,j+\frac{1}{m}): m\geq 1\rangle$. Pick a bound $b_j>b_{j-1}$ for $H(\tau^j)$: by the assumption on $K$, there exists $i_j>i_{j-1}$ such that
        \[K\langle b_j, \langle \tau_0^j,\dots,\tau_{i_j}^j\rangle  \rangle \downarrow \;=\langle\: \alpha^j,\: (0,j+\frac{1}{m_0^j}),\dots,\: (0,j+\frac{1}{m_{h_j}^j}) \:\rangle\]
        for some $\alpha^j\subseteq\langle (1,1),\dots, (1,\frac{1}{j+1})\rangle$ and some integers $h_j\geq 0$, $1\leq m_{0}^j<\dots \:<m_{h_j}^j$.\\
        Finally, set $\sigma:=\langle\sigma_0,\sigma_1,\dots \rangle$, where $\sigma_i:= \tau_i^{j}$ if $0\leq i\leq i_{j}$. Notice that $\sigma$ is well-defined by clause (\ref{indu}). Moreover, if $j$ is sufficiently large then $b_j$ is a bound to $H(\sigma)$. Now $K$ outputs $\langle\alpha^j,  (0,j+\frac{1}{m_0^j}),\dots, \: (0,j+\frac{1}{m_{h_j}^j}) \:\rangle$ after having read the integer $b_j$ together with the rows $\sigma_0=\tau_0^j,\dots,\sigma_{i_j}=\tau_{i_j}^j$ of $\sigma$. This contradicts the assumption on $K$, because the solutions to $\WOP(X\mapsto X_{lex}^2)$ for $\sigma$ are exactly the subequences of $\langle (1, \frac{1}{m})  : m\geq 1\rangle $.\\
    \end{proof}
\end{proposition}

\begin{corollary}
    $\WOP(X\mapsto X^\omega) \nleq_\W \ECT$.
\begin{proof}
    Combine proposition \ref{wop_ect} with the fact that $\WOP(X\mapsto X_{lex}^2)\leq_{(s)\W} \WOP(X\mapsto X^\omega)$
\end{proof}
\end{corollary}

In order to prove that $\ECT \leq \WOP(X \mapsto X^\omega)$ we need the following proposition. 

\begin{proposition}
\label{tcn_wop}
   $(\TC_{\mathbb N})^n\leq_{\sW} \WOP(X\mapsto X_{lex}^{2^n})$ for all $n\geq 1$.
\end{proposition}

For our proof of Proposition \ref{tcn_wop} we need the following definitions and lemma.

    \begin{definition}
    Let $(U_i)_{i\geq 0}$ be any sequence of subsets of $[n]$. Define a set of $\oplus_{i}{U_i}$-computable sequences $reset_A$, $A\subseteq [n]$, as follows. Start by setting $reset_A(0):=-1$ for all $A\subseteq[n]$. Let $i\geq 0$ and assume that, for all $A\subseteq[n]$, $reset_A(i)$ has been defined and $reset_A(i)< i$. We say that $B\subseteq [n]$ completes a cycle at stage $i$ if 
    \[\bigcup_{reset_B(i)<h\leq i}U_h= B.\]
    For all $A\subseteq [n]$, set: $reset_A(i+1):= i$ if some  $B\supseteq A$  completes a cycle at stage $i$, in which case we say that $A$ is reset at stage $i$; $reset_A(i+1):=reset_A(i)$ otherwise.
    \end{definition}
    
    \begin{lemma}
    \label{vi}
    For all $i\geq 0$, 
    \[ V_{i}:=\bigcup_{reset_{U_i}(i)<h\leq i}U_h.\]
    is the unique subset of $[n]$ completing a cycle at stage $i$.
    \begin{proof}
    It's clear that $V_0=U_0$ is the unique subset of $[n]$ completing a cycle at stage $0$.\\
    Let $i\geq 1$. Assume by ($\Delta_0^{0, \oplus_i{U_i}}$-) induction that, for all $0\leq h<i$, $V_h$ is the unique subset of $[n]$ completing a cycle at stage $h$. This implies that $A\subseteq [n]$ is reset at stage $0\leq h<i$ if and only if $A\subseteq V_h$.\\
    First, we show that if $reset_{U_i}(i)\geq 0$ then  $V_i\subseteq V_{reset_{U_i}(i)}$. Suppose otherwise, and let $reset_{U_i}(i)<  h\leq   i$ be least such that $U_h\nsubseteq V_{reset_{U_i}(i)}$. Since $U_i$ is reset at stage $reset_{U_i}(i)$, it must be contained in $V_{reset_{U_i}(i)}$ by the inductive hypothesis. Hence $h<i$. By the choice of $h$, $U_l\subseteq V_{reset_{U_i}(i)}$ is also reset at stage $reset_{U_i}(i)$ for all $reset_{U_i}(i)<l<h$. Then, for each $reset_{U_i}(i)<k<h$, $V_k\subseteq \bigcup_{reset_{U_i}(i)<l\leq k}U_{l}\subseteq  V_{reset_{U_i}(i)}$, so that $U_h\nsubseteq V_{k}$ is not reset at stage $k$ by the inductive hypothesis. We infer that $reset_{U_h}(h)< reset_{U_i}(i)$. This implies that $U_{reset_{U_i}(i)}\subseteq V_h$, which in turn implies that $reset_{V_h}(reset_{U_i}(i))\leq reset_{U_{reset_{U_i}(i)}}(reset_{U_i}(i))$. Since $reset_{V_h}(reset_{U_i}(i))=reset_{V_h}(h)$, it follows that $V_{h}\supseteq V_{reset_{U_i}(i)}$ ($\supsetneq$, actually). In particular, $U_i\subseteq V_{reset_{U_i}(i)}$ is reset at stage $h\in (reset_{U_i}(i),i)$, which contradicts the definition of $reset_{U_i}(i)$.\\ 
    We show that $V_{i}$ completes a cycle at stage $i$. This is clearly the case if $reset_{U_i}(i)=-1$. Suppose instead that $reset_{U_i}(i)\geq 0$. Surely $reset_{V_{i}}(i)\leq reset_{U_i}(i)$, as $V_{i} \supseteq U_i$. It then suffices to show that $V_i$ is reset at stage $reset_{U_i}(i)$. By the inductive hypothesis, this is the same as saying that $V_i\subseteq V_{reset_{U_i}(i)}$, which was proven above.\\
    Lastly, we prove that $V_i$ is unique. Suppose by way of contradiction that $W\neq V_i$ completes a cycle at stage $i$. Then $W\supseteq U_i$, hence $reset_{W}(i)\leq reset_{U_i}(i)=reset_{V_i}(i)$. Since both $W$ and $V_i$ complete a cycle at stage $i$, this implies that $W\supseteq V_i$, or rather that $W\supsetneq  V_i$, by our assumption. We must then have $reset_W(i)< reset_{U_i}(i)$ and $W-V_i\subseteq \bigcup_{reset_{W}(i)<h\leq reset_{U_i}(i)} U_h$. The former gives $W\supseteq U_{reset_{U_i}(i)}$, so that $reset_W(i)=reset_W(reset_{U_i}(i))\leq reset_{U_{reset_{U_i}(i)}}(reset_{U_i}(i))$. It follows that
    \[ V_{reset_{U_i}(i)}\subseteq \bigcup_{reset_{W}(i)<h\leq reset_{U_i}(i)} U_h.\]
    Therefore
    \[W=(W-V_i)\cup V_i \subseteq(W-V_i) \cup V_{reset_{U_i}(i)}  \subseteq \bigcup_{reset_{W}(i)<h\leq reset_{U_i}(i)} U_h\subseteq W.\]
    We conclude that $W$ completes a cycle at some stage $h\in (reset_W(i), i)$, which contradicts the definition of $reset_W(i)$.
    \end{proof}
    \end{lemma}
    \begin{proof}[Proof of proposition \ref{tcn_wop}]
    Let $\overrightarrow e=\langle e_1,\dots, e_n\rangle$ be an instance of $(\TC_{\mathbb N})^n$.\\
    For a fixed $m\in [n]$, $g_m(-1):=0$ is the default, optimistic guess for a solution to $\TC_{\mathbb N}$ for $e_m$. At each stage $i\geq 0$, the $(i-1)$-th guess is updated to $g_m(i):=\min \big(\mathbb N- ran (e_m\restriction \{0,\dots,i\})\big)$. Let $\overrightarrow g$ be the sequence of $n$-tuples $\langle g_1(i),\dots, g_n(i)\rangle$, $i\geq -1$. Then the set of  components of $\overrightarrow g$ that are updated at each stage $i\geq 0$ is
    \[U_i:=\{m\in [n]\;:\;e_m\text{ enumerates } g_{m}(i-1) 
    \text{ at stage } i\}.\]
    Both $\overrightarrow g$ and the cycle structure associated to its update sets $(U_i)_{i\geq 0}$ are encoded into an $\overrightarrow e$-computable sequence $\sigma=(\sigma_i)_{i\geq 0}$ in $[0,\infty ]^{2^n}$. To build $\sigma$, first fix a total ordering $A_{1}=\emptyset<A_2<\dots \:<A_{2^n}=[n]$ of $\mathcal P([n])$ extending $\subseteq$ and, for each $2^n\geq j\geq 1$, initialize a sequence $count_j$ by setting $count_j(0):=2^n-|A_j|$.\\
    Stage $0$. Set:
    \[a_{j_0}^{0}:=count_{j_0}(0)+\frac{1}{p_1^{g_1(0)}p_2^{g_2(0)} \dots\: p_n^{g_n(0)}},\] where $A_{j_0}=U_0$ and $p_m$ denotes the $m$-th prime number; 
    \[a_j^{0}:=\infty \;\text{for all}\; j\neq j_0;\] $\sigma_0:=(a_{2^n}^0,\dots \:,a_{1}^0)$. For all $2^n\geq j\geq 1$, set:
    \[
    count_j(1):=\begin{cases}
			count_{j_0}(0)+ |A_{j_0}- A_j|, & \text{if $A_j\subsetneq A_{j_0}$}\\
            count_j(0), & \text{otherwise}.
		 \end{cases}
    \]
    Stage $i\geq 1$. Let $j_i$ satisfy $A_{j_i}=V_i$, where the latter is defined as in lemma \ref{vi}. Set:  
    \[ a_j^{i}:=a_j^{i-1}\;\;\text{for all}\; j>j_i;\]
    \[ a_{j_i}^{i}:=count_{j_i}(i) +\frac{1}{p_1^{g_1(i)}p_2^{g_2(i)}\dots \: p_n^{g_n(i)}p_{n+1}^{i}}; \]
    \[ a_j^{i}:=\infty \;\;\text{for all}\; j_i>j\geq 1;\]
    $\sigma_i:=(a_{2^n}^i,\dots,a_{1}^i)$. For all $2^n\geq j\geq 1$, set:
    \[
    count_j(i+1):=\begin{cases}
			\max_{0\leq h\leq i}count_{j_h}(h)+ |A_{j_i}- A_j|, & \text{if $A_j\subsetneq A_{j_i}$}\\
            count_j(i), & \text{otherwise}.
		 \end{cases}
    \]
    This completes the definition of $\sigma$.\\
    In order to show that $\sigma$ is an instance of $\WOP(X\mapsto X^{2^n}_{lex})$, i.e. that it's descending with respect to $<_{lex}$, we only have to check that $a_{j_i}^{i}<a_{j_i}^{i-1}$ for all $i\geq 1$. This is obviously the case if $a_{j_i}^{i-1}=\infty$. Now $a_{j_i}^{i-1}<\infty$ if and only if there exists $0\leq i'<i$ such that $j_{i'}=j_i$ and $A_{j_h}\nsupseteq A_{j_i}$ for all $i'<h< i$. This implies that $count_{j_i}(i')=count_{j_i}(i)$. Moreover, $g_m(i')\leq g_m(i)$ for all $m\in [n]$  (with $g_m(i')< g_m(i)$ if and only if $m\in A_{j_i}$), and $p_{n+1}^{i'}<p_{n+1}^{i}$. The claim follows.
    \begin{remarks}
    \begin{enumerate}
        Let $2^n\geq y\geq 1$ satisfy
        \[A_y=\{m\in [n]\:|\:ran(e_m)=\mathbb N \}.\]
        Notice that, for all $m\in[n]$, $m\in A_y$ if and only if $m\in U_i$ for infinitely many $i$-s.
        \item \label{one} Of course, if $m\in A_y$ then any natural number is a solution to $TC_{\mathbb N}$ for $e_m$. Let $i_*\geq -1$ be least such that $U_i\subseteq A_y$ for all $i>  i_*$. Then, for all $m\notin A_y$, $g_m(i_*)$ is a (the least) solution to $\TC_{\mathbb N}$ for $e_m$ and $g_m(i)=g_m(i_*)$ for all $i\geq i_*$.
        \item \label{two} Let $i_{**}\geq -1$ be least such that $A_{j_i}\subseteq  A_y$ for all $i> i_{**}$. If $i_*=-1$ then $i_{**}=-1$ as well. Otherwise $U_{i_*}\nsubseteq A_y$, hence $A_{j_{i_*}}\nsubseteq A_y$: in this case, $i_{**}$ is the unique $i\geq i_*$ such that $A_{j_{i}}\supsetneq A_{y}$.
        \item \label{three} Let $i_{0}$ be the least $i> i_{**}$ such that $y=j_i$. For all $k\geq 1$, let $i_{k}$ be the least $i> i_{k-1}$ such that $y=j_i$. Then $a_y^i=\infty$ for all (possible) $i_{**}<i<i_0$, and
        \[\infty> a_y^{i_0}=\;...\;=a_{y}^{i_1-1}>a_{y}^{i_1}=\dots \:>a_{y}^{i_{k-1}}=\:...\:=a_{y}^{i_{k}-1}>a_{y}^{i_{k}}=\dots \;.\]
        \item \label{four}Suppose that $i_{**}\geq 0$. Let $0\leq i\leq i_{**}$ and $2^n\geq j\geq 1$. By remark \ref{two}, either $a_j^i=\infty$ or $a_j^i<a_y^{i_k}$ for all $k\geq 0$.
        \item \label{five} For all $y>j\geq 1$ and for all $i>i_{**}$, either $a_j^i=\infty$ or $a_j^i=a_{j_{h}}^{h}$ for some $i_{**}<h\leq i$. In both cases, $a_j^i>a_y^{i_0}$.
        \item \label{six} For all $2^n\geq j\geq 1$ such that $A_j\nsubseteq A_y$ (hence for all $2^n\geq j>y$) and for all $i>  i_{**}$, $a_j^i=a_j^{i_{**}\vee 0}$.
    \end{enumerate}
     \end{remarks}

    Let $\overrightarrow x$ be a solution to $\WOP(X\mapsto X^\omega)$ for $\sigma$. It follows from remarks \ref{three}, \ref{four}, \ref{five}, \ref{six} above that 
    \[\overrightarrow x=\alpha\overrightarrow v\overrightarrow w,\]
    where: $\alpha$ is either $\infty$ or the empty sequence; $\overrightarrow v$ is a (necessarily finite, and possibly empty) descending sequence contained in $\{a_{j_i}^i\;:\;i> i_{**}\:\&\:y>j_i\}$; $\overrightarrow w$ is a subsequence of $(a_y^{i_k})_{k\geq 0}$. Pick any finite term of $\overrightarrow x$, and compute the corresponding guess $\langle g_1(i),\dots \:,g_n(i)\rangle$, $i>i_{**}\geq i_*$. This is a solution to $(\TC_{\mathbb N})^n$ for $\overrightarrow e$ by remark \ref{one}.
\end{proof}

\begin{corollary}
    $\ECT\leq_{\W} \WOP(X\mapsto X^\omega)$.
    \begin{proof}
        By Proposition \ref{tcn_wop}, $(\TC_{\mathbb N})^n\leq_{\sW} \WOP(X\mapsto X_{lex}^{2^n})\leq_\W \WOP(X\mapsto X^\omega)$ for all $n\geq 1$. Hence $ (\TC_{\mathbb N})^*\leq_\W \WOP(X\mapsto X^\omega)$. The desired conclusion follows from the fact that $\ECT \equiv_\W (\TC_\Nat)^*$ (Lemma \ref{lem:ect_tc}).
    \end{proof}
\end{corollary}

Note that, since the proof of $(\TC_\Nat)^*\equiv_\W \ECT$ in \cite{DHHPPK:20} shows $(\TC_\Nat)^n \equiv_\W \ECT_{n+1}$, 
where $\ECT_{n+1}$ is the restriction of $\ECT$ to coloring in $(n+1)$ colors, Proposition \ref{tcn_wop} also shows
$\ECT_{n+1}\leq_{\W} \WOP(X\mapsto X_{lex}^{2^n})$ for all $n\geq 1$.

\section{Conclusion and perspectives}\label{sec:conclusion}

We have compared in terms of Weirauch reducibility the Ordered Ramsey Theorem $\ORT_\Nat$ and the Well-Ordering Preservation Principle for $\omega$-exponentiation $\WOP(X \mapsto X^\omega)$. The fact that the Ordered Ramsey Theorem is reducible to the parallel product of $\ECT$ and $\WOP(X\mapsto X^\omega)$ is the first example of a Weihrauch reduction of a combinatorial principle to a well-ordering preservation principle. 

While we obtained a complete picture concerning the principles studied in this paper, many question remains for other principles of interest.
We consider the present work as a specimen of a wider study of Weihrauch reducibilities between WOPs and Ramsey-type theorems. 
We believe that the approach used in the present paper can be successfully applied to many other principles of strength around that of $\Sigma^0_2$-induction. We believe that the reductions from WOPs can be useful to obtain analogous results also at higher level of logical complexity, e.g., for the principles studied in \cite{Car-Mai-Zda:25}.

\end{document}